\def\mytitle{Optimized State Space Grids for Abstractions}
\def\myname{
\ifx\shortnames\undefined Alexander \fi
Weber%
,
\ifx\shortnames\undefined Matthias \fi
Rungger%
~and
\ifx\shortnames\undefined Gunther \fi
Reissig%
}
\def\mykeywords{Discrete abstraction, symbolic control, automated
  synthesis, Djokovi{\'c}-London functional%
\ifCLASSOPTIONonecolumn;
MSC: Primary, 93B51;
Secondary, 93B52, 93C10, 93C30, 93C55, 93C57, 93C65%
\fi}
\let\confidentialstring=\relax
\let\headnote=\relax
\def\arxivVersion{}
\ifx\arxivVersion\undefined
\documentclass[a4paper]{IEEEtran}
\relax\else
\date{}
\documentclass[journal,a4paper,onecolumn,12pt,nofonttune]{IEEEtran}
\fi
\let\confidentialstring=\relax

\usepackage[normalem]{ulem}
\def\emph#1{\textit{#1}}
\renewcommand\sout{\bgroup\markoverwith
{\textcolor{red}{\rule[.5ex]{2pt}{1pt}}}\ULon}

\newcommand\soutAW{\bgroup\markoverwith
{\textcolor{cyan}{\rule[.5ex]{2pt}{1pt}}}\ULon}

\usepackage{amssymb}
\usepackage{cite}
\usepackage{pgf}
\usepackage{psfrag}
\usepackage{xcolor}

\usepackage{svn-multi}
\svnid{$Id: i15grid.tex 970 2016-12-15 09:09:34Z lf1balwe $}

\usepackage[cmex10]{amsmath}
\let\ORGforeignlanguage\foreignlanguage
\def\foreignlanguage#1{\lowercase{\ORGforeignlanguage{#1}}}
\makeatletter
\def\MakeUppercase#1{#1}%
\def\markboth#1#2{\def\leftmark{\@IEEEcompsoconly{\sffamily}\MakeUppercase{#1}}%
\def\rightmark{\@IEEEcompsoconly{\sffamily}\MakeUppercase{#2}}}
\makeatother
\ifCLASSOPTIONonecolumn

\fi

\usepackage{GR/gunther2e}

\ifx
\hypersetup\undefined\def\href#1#2{\texttt{#2}}
\else
\hypersetup{
pdftitle={\mytitle},
pdfauthor={Alexander Weber, Matthias Rungger, Gunther Reissig, http://www.reiszig.de/gunther/},%
pdfsubject={\confidentialstring{} \headnote},
pdfkeywords={\mykeywords}
}
\fi

\ifx\arxivVersion\undefined
\usepackage{multibib}
\newcites{review}{References}
\else\relax
\fi

\begin{document}\bstctlcite{IEEEtranBSTCTL_AuthorNoDash}%

\makeatletter
\renewenvironment{proof}[1][\proofname]{\par
  \pushQED{\qed}%
  \normalfont \topsep6\p@\@plus6\p@\relax
  \trivlist
  \item[\hskip\labelsep
        \itshape
    #1\@addpunct{.}]\ignorespaces
}{%
  \popQED\endtrivlist\@endpefalse
}
\markboth{\def\shortnames{}\myname\hspace*{\fill}\def\\{ }\mytitle\ifx\arxivVersion\undefined\hspace*{\fill}\@date\else\relax\fi\hspace*{\fill}\ifx\DraftVersion\undefined\relax\else\hspace*{\fill}svn: \svnrev\fi\hspace*{\fill}}%
{\def\shortnames{}\myname\hspace*{\fill}\def\\{ }\mytitle\ifx\arxivVersion\undefined\hspace*{\fill}\@date\else\relax\fi\hspace*{\fill}\ifx\DraftVersion\undefined\relax\else\hspace*{\fill}svn: \svnrev\fi\hspace*{\fill}}%
\makeatother

\newcommand{\R}{\mathbb{R}}
\newcommand{\N}{\mathbb{N}}
\newcommand{\Z}{\mathbb{Z}}
\newcommand{\Q}{\mathbb{Q}}

\title{%
\mytitle}

\author{\myname%
\thanks{%
\ifx\arxivVersion\undefined
Corresponding author: A. Weber.\newline
\else\relax
\fi
G. Reissig and A. Weber are with the
University of the Federal Armed Forces Munich,
Dept. Aerospace Eng.,
Chair of Control Eng. (LRT-15),
D-85577 Neubiberg (Munich),
Germany,
\ifx\DraftVersion\undefined%
\url{http://www.reiszig.de/gunther/}, \url{a.weber@unibw.de}%
\else%
\url{no-email-gunther}, \url{A.Weber@unibw.de}%
\fi%
}%
\thanks{%
M. Rungger is with the Hybrid Control Systems Group at the Department of
Electrical and Computer Engineering at the Technical University of Munich, Germany, \url{matthias.rungger@tum.de}%
}%
\thanks{This work has been supported by the German Research Foundation (DFG) under grants no. RE 1249/3-2 and RE \mbox{1249/4-1}.
\ifx\arxivVersion\undefined\relax\else
This is the accepted version of a paper published in \emph{IEEE Trans. Automat. Control}, vol.~62, no.~11, pp.~5816-5821, 2017, \href{http://dx.doi.org/10.1109/TAC.2016.2642794}{DOI:10.1109/TAC.2016.2642794}.
\fi
}%
\ifCLASSOPTIONonecolumn%
\thanks{\confidentialstring{} \headnote}%
\fi%
}

\maketitle

\begin{abstract}
\noindent
The practical impact of abstraction-based controller synthesis methods
is currently limited by the immense computational effort for obtaining
abstractions.
In this note we focus on a recently proposed method to compute
abstractions whose state space is a cover of the state space of the
plant by congruent hyper-intervals. 
The problem of how to choose the size
of the hyper-intervals so as to obtain 
computable and useful abstractions is unsolved.
This note provides a twofold contribution towards a solution.
Firstly, we present a functional to predict the computational
effort for the abstraction to be computed. 
Secondly, we propose a method for choosing the aspect ratio of the hyper-intervals
when their volume is fixed. More precisely,
we propose to choose the aspect ratio
so as to minimize a predicted number of transitions of
the abstraction to be computed, in order to reduce the computational
effort.
To this end, we derive a functional to predict the number of
transitions in dependence of the aspect ratio. The functional is to be
minimized subject to suitable constraints. We characterize the unique
solvability of the respective optimization problem and prove that it
transforms, under appropriate assumptions, into an equivalent convex
problem with strictly convex objective. The latter problem can then be
globally solved using standard numerical methods.
We demonstrate our approach on an example.\looseness=-1
\end{abstract}

\section{Introduction}
\label{s:intro}
The concept of abstraction-based controller synthesis
is a fully automated procedure to design feedback controllers
that enforce predefined, possibly complex, specifications
on nonlinear control systems \cite{i14sym,Tabuada09}. 
The procedure comprises three steps \cite{Tabuada09}. 
The first step is to transfer the actual control system (``plant") 
together with the predefined specification
to an auxiliary control system, 
known as \begriff{abstraction} or \begriff{symbolic model},
and an auxiliary specification.
In the second step, the auxiliary control problem 
is solved. 
The last step is to refine the obtained controller (``abstract controller") 
to a controller for the actual control problem. 
The practical impact of the approach is currently limited by
the immense computational effort for the first step, i.e., for
obtaining abstractions.

Various methods to reduce the computational effort of 
this procedure exist in literature, e.g. 
\cite{
RunggerStursberg12,
RunggerMazoTabuada13,
PolaBorriDiBenedetto12,
GirardGoesslerMouelhi16,
TazakiImura09,
MouelhiGirardGossler13,
LeCorroncGirardGoessler13,
ZamaniTkachevAbate14}.
The methods in 
\cite{
RunggerStursberg12,
RunggerMazoTabuada13,
PolaBorriDiBenedetto12,
GirardGoesslerMouelhi16} 
merge the first and second step in previous scheme 
in order to compute the abstraction only partially.
The methods in \cite{TazakiImura09,MouelhiGirardGossler13,GirardGoesslerMouelhi16} 
locally refine symbolic models to reduce the number of required abstract states. 
In \cite{LeCorroncGirardGoessler13,ZamaniTkachevAbate14} 
the state space of the plant is not discretized but
finite sequences of inputs are used as abstract states. 

This paper is the first to establish
a reduction method for abstractions that are based on 
\emph{feedback refinement relations} \cite{i14symc,i14sym}.
Moreover, for the first time, a functional is presented that 
\emph{predicts} the required computational resources for the 
abstraction to be computed. 

Abstractions based on feedback refinement relations 
can be constructed for plants whose dynamics are 
governed by nonlinear differential equations 
subject to perturbations. 
Moreover, in contrast to other system relations, 
the induced controllers for 
the actual control problem merely consist
of the abstract controller and a static quantizer \cite{i14sym}. 

The scheme for computing such an abstraction is as follows.
First, the $n$-dimensional real state space of the plant is
discretized by means of a cover to obtain the states of the
abstraction, where the vast majority of the elements of the cover are
translated copies of the hyper-interval 
\begin{equation}
\label{e:hyperinterval}
\intcc{-\tfrac{\eta_1}{2},\tfrac{\eta_1}{2}} \times \ldots \times \intcc{-\tfrac{\eta_n}{2},\tfrac{\eta_n}{2}},
\quad \eta_1,\ldots,\eta_n>0,
\end{equation}
which are aligned on a uniform grid.
Second, attainable sets of the sets in the cover 
are over-approximated by hyper-intervals
to obtain the transitions in the abstraction. 

The goal of this work is to provide a heuristic for 
choosing $\eta_1,\ldots, \eta_n$ in \ref{e:hyperinterval} so as 
to reduce the memory and time consumption when computing abstractions
for which the volume $\eta_1\!\cdot\!\eta_2\ldots\eta_n$ of \ref{e:hyperinterval}
is predefined. The key idea here is the minimization of 
the expected number of transitions. 
As a first step towards this goal,
we propose to use the functional 
\begin{equation}
\label{e:DLFunction}
E(\eta) = \prod_{i=1}^n\nolimits \frac{1}{\eta_i}\left (p_i+ \sum_{j=1}^n\nolimits A_{i,j}\eta_j \right )
\end{equation}
to estimate the number of transitions per abstract state and input symbol in dependence 
of $\eta = (\eta_1,\ldots, \eta_n)$ in \ref{e:hyperinterval}.
Here, the nonnegative $n\times n$-matrix $A$ and the $n$-dimensional nonnegative
vector $p$
depend on the particular plant dynamics and on bounds on disturbances.

In the next step we study
the minimization of \ref{e:DLFunction} subject to
a constraint that
prescribes the volume of \ref{e:hyperinterval}.
For this, in general, non-convex optimization problem, 
we characterize the existence of a unique solution.
To this end, we eliminate non-convexity by suitably
transforming \ref{e:DLFunction} and show that under appropriate assumptions the auxiliary
optimization problem has strictly convex objective and can be globally
solved by standard numerical methods. These results then allow us to
establish the requested heuristic.
We finally demonstrate our approach on an example.

Our results on the functional in \ref{e:DLFunction} recover and extend
previous results for the special case $p=0$, which plays a part in
diagonal scaling of nonnegative matrices into doubly stochastic form
\cite{Djokovic70,London71}.
Some of our results have been announced in \cite{i15gridc}.

\section{Preliminaries}
\label{s:prelims}
\subsubsection{Notation}
$\mathbb{R}$, $\mathbb{R}_+$ and $\mathbb{Z}$ denote the sets of
real numbers, nonnegative real numbers, and integers,
respectively. 
$\intcc{a,b}$, $\intoo{a,b}$,
$\intco{a,b}$, and $\intoc{a,b}$
denote closed, open and half-open, respectively,
intervals with end points $a$ and $b$.
$\intcc{a;b}$, $\intoo{a;b}$,
$\intco{a;b}$, and $\intoc{a;b}$ stand for discrete intervals, e.g.
$\intcc{a;b} = \intcc{a,b} \cap \mathbb{Z}$,
$\intco{1;4} = \{ 1,2,3\}$, and
$\intco{0;0} = \emptyset$.
In $\mathbb{R}^n$, the relations $<$, $\leq$, $\geq$, $>$ are defined
component-wise, e.g.,
$a < b$ iff $a_i < b_i$ for all $i \in \intcc{1;n}$. For $x \in \R^n$ we define $|x| = (|x_1|,\ldots,|x_n|)$. 
For $a, b \in (\mathbb{R} \cup \{\infty,-\infty\})^n$, $a \leq b$, the closed
hyper-interval $\segcc{a,b}$ is defined by
$
\segcc{a,b}
=
\mathbb{R}^n
\cap
\left(
\intcc{a_1,b_1} \times \cdots \times \intcc{a_n,b_n}
\right)
$.
$\innerProd{\cdot}{\cdot}$ stands for the standard Euclidean inner product, 
i.e., $\innerProd{x}{y} = \sum_{i=1}^n x_iy_i$. 
$\| \cdot \|_p$ stands for the usual $p$-norm, $p \in
\intcc{1,\infty}$.\looseness=-1

$f \colon A \rightrightarrows B$ denotes a set-valued map of
$A$ into $B$, whereas $f \colon A \to B$ denotes an ordinary map; see \cite{RockafellarWets09}.
If $f$ is
set-valued, then $f$ is \emph{strict} 
if
$f(a) \not= \emptyset$
for every $a$.
We identify set-valued maps
$f \colon A \rightrightarrows B$ with binary relations on
$A \times B$, i.e., $(a,b) \in f$ iff $b \in f(a)$.
$f \circ g$ denotes the composition of $f$ and $g$, $(f \circ g)(x) = f(g(x))$. 

We denote the vector $(1,\ldots,1) \in \R^k$  by $\mathbf{1}$ and the identity map
$X \to X \colon x \mapsto x$ by $\id$. 
The dimension $k$ and the domain of definition $X$ respectively 
will always be clear from the context. 
A \begriff{cover} of a set $X$ is a set of subsets of $X$ whose union equals $X$. 
\subsubsection{nonnegative matrices}
A matrix $A$ is \begriff{nonnegative} if 
$A \in \R^{n\times n}_+$ and it
is \begriff{essentially nonnegative} if 
$A \in \R^{n\times n}$ and $A_{i,j}\geq 0$ whenever $i \neq j$. 
A matrix $A \in \R^{n\times n}$ is \begriff{irreducible} if 
for any $r,s \in \intcc{1;n}$, $r \neq s$ 
there exist distinct indices $i_1,\ldots,i_m \in \intcc{1;n}$ 
satisfying $i_1 = r$, $i_m = s$ and $A_{i_{k},i_{k+1}}>0$
for all $k \in \intco{1;m}$. Otherwise, $A$ is \begriff{reducible}.
\section{Computation of abstractions}
\label{s:abstractions}

This section gives a brief exposition of the method to 
compute abstractions from \cite{i14sym}.
We consider control systems
governed by nonlinear differential inclusions of the form 
\begin{equation}
\label{e:System:c-time}
\dot \xi(t) \in f(\xi(t),u) + \segcc{-w,w},
\end{equation}
where 
$f \colon \mathbb{R}^n \times \bar U \to \mathbb{R}^n$,
$\bar U\subseteq \mathbb{R}^m$ is nonempty,
$f(\cdot,u)$ is locally Lipschitz for all $u \in \bar U$,
and
$w \in \R^n_+$ is a component-wise bound on perturbations 
to the dynamics of the control system.
For $\tau > 0$ a \begriff{solution of \ref{e:System:c-time} on $\intcc{0,\tau}$ 
with (constant) input $u \in \bar U$} is 
an absolutely continuous function $\xi \colon \intcc{0,\tau} \to \R^n$ 
that fulfills \ref{e:System:c-time} 
for almost every $t\in \intcc{0,\tau}$ \cite{Filippov88}.

We formalize 
sampled versions of control systems \ref{e:System:c-time} 
in a notion of system as given below.
\begin{definition}
\label{d:system}
A \begriff{system} is a triple $(X,U,F)$, 
where $X$ and $U$ are nonempty sets and 
$F \colon X \times U \rightrightarrows X$.
\end{definition}
We call the sets $X$ and $U$ the \begriff{state} 
and \begriff{input alphabet}, respectively. 
The map $F$ is called the \begriff{transition function}.
\begin{definition}
\label{def:SampledSystem}
Let $S=(X,U,F)$
be a system and
\mbox{$\tau > 0$.}
We say that $S$ is the \begriff{sampled system} associated with the
control system \ref{e:System:c-time} and the \begriff{sampling time}
$\tau$, if $X=\mathbb{R}^n$, $U=\bar U$ and
the following holds:
$x_1 \in F(x_0,u)$ iff there exists a solution $\xi$ of \ref{e:System:c-time} on $\intcc{0,\tau}$ with input $u$ so that $\xi(0) = x_0$
and $\xi(\tau) = x_1$.
\end{definition}
We relate two systems to each other by feedback refinement relations. 
We introduce this concept as follows. 

For a system $S = (X,U,F)$ and $x \in X$ let 
\begin{equation*}
U_{S}(x) = \{ u \in U \ | \ F(x,u) \neq \emptyset \}.
\end{equation*}
\begin{definition}
\label{d:FRR}
Let $S_i = (X_i,U_i,F_i)$, $i \in \{1,2\}$ be two systems such that 
$U_2 \subseteq U_1$. 
A \begriff{feedback refinement relation from $S_1$ to $S_2$} is 
a strict relation $Q \subseteq X_1 \times X_2$ satisfying 
\begin{enumerate}
\item \label{d:FRR:in} $U_{S_2}(x_2) \subseteq U_{S_1}(x_1)$,
\item \label{d:FRR:dyn} $u \in U_{S_2}(x_2) \ \Longrightarrow \ Q(F_1 (x_1,u))\subseteq F_2(x_2,u)$
\end{enumerate}
for all $(x_1,x_2) \in Q$.
\end{definition}
We write $S_1 \preccurlyeq_Q S_2$ if
$Q$ is a feedback refinement relation from $S_1$ to $S_2$. 
If $S_1 \preccurlyeq_Q S_2$ we say that $S_2$ is an \begriff{abstraction} for $S_1$.\looseness=-1

A feedback refinement relation $Q$ from $S_1$ to $S_2$ associates 
states of $S_1$ with states of $S_2$ and imposes conditions 
on the images of the transition functions at associated states.
The relation $Q$ also serves as an interface to be added to the
abstract controller in order to refine it into a controller for the
actual plant.
We refer the reader to
\cite{i14sym} for a formal definition of the closed loop, the details
of the synthesis procedure, and in particular, for
a proof of the fact that the refined controller actually solves the control
problem for the plant $S_1$.
The framework also allows for bounded measurement errors
$P \colon \R^n \rightrightarrows \R^n$ of the form
\begin{IEEEeqnarray}{c}\label{e:cSys:meas}
P(x)=x+\segcc{-z,z}
\end{IEEEeqnarray}
for some $z \in \R^n_+$
which are taken care of
by simply requiring $S_1 \preccurlyeq_{Q \circ P} S_2$ rather than
$S_1 \preccurlyeq_{Q} S_2$
\cite[Sec. VI.B]{i14sym}. 

In what follows, we discuss the computation of abstractions
$S_2 = (X_2,F_2,U_2)$ satisfying $S_1 \preccurlyeq_{Q \circ P} S_2$,
for a sampled system $S_1$ associated with \ref{e:System:c-time},
where we restrict our attention to abstractions whose state alphabet
$X_2$ is a cover of the state alphabet of $S_1$.
The elements of $X_2$ are nonempty, closed hyper-intervals, which we call \begriff{cells}. 
We divide $X_2$ into two subsets, which we interpret as ``real'' quantizer symbols 
and 
overflow symbols, respectively. 
See~\cite[Sec. III.A]{i11abs}. 
We let the former subset, subsequently denoted by $\bar X_2$, consist of congruent cells that are
aligned on the uniform grid 
\begin{IEEEeqnarray}{c}
\label{e:grid}
\eta\Z^n=\{c\in \R^n\mid \exists_{k\in\Z^n}\forall_{i\in\intcc{1;n}}\; c_i=k_i\eta_i\}
\end{IEEEeqnarray}
with \begriff{grid parameter} $\eta\in(\R_+\setminus\{0\})^n$, i.e.,
\begin{IEEEeqnarray}{c}\label{e:abs:ss}
x_2\in \bar X_2\implies \exists_{c\in \eta\Z^n}\;x_2=c+\segcc{-\eta/2,\eta/2}.
\IEEEeqnarraynumspace
\end{IEEEeqnarray}
The computation of the map $F_2$ on $\bar X_2 \times U_2$ will be based on 
overapproximating the attainable sets of the cells in $\bar X_2$ under the flow of \ref{e:System:c-time}.
For that purpose, we will define \begriff{growth bounds} below.
Growth bounds have been introduced in 
\cite[Sec.~VIII.A]{i14sym}, where their important features are also discussed.\par
We denote by $\varphi$ the general solution of the unperturbed control system 
associated with \ref{e:System:c-time}.
More formally, if $x_0 \in \R^n$, $u \in \bar U$, then $\varphi(t, x_0, u)$ is the value at time $t$ 
of the solution of the initial value problem $\dot x = f(x,u)$, $x(0) = x_0$. 
\begin{definition}
Let $\tau > 0$, $K \subseteq \R^n$, and $\bar U' \subseteq \bar U$. A map $\beta \colon \mathbb{R}^n_+ \times \bar U' \to \mathbb{R}^n_+$ is a \begriff{growth bound} on $K$, $\bar U'$ associated with $\tau$ and \ref{e:System:c-time} if
\begin{enumerate}
\item
\label{def:GrowthBound:beta}
$\beta(r,u) \geq \beta(r',u)$ whenever $r \geq r'$ and $u \in  \bar U'$,
\item
\label{def:GrowthBound:bound}
$\intcc{0,\tau} \times K \times \bar U' \subseteq \dom \varphi$ and if $\xi$ is a solution of \ref{e:System:c-time} on 
$\intcc{0,\tau}$ with input $u \in \bar U'$ and $\xi(0),p \in K$
then
\mbox{$| \xi(\tau) - \varphi(\tau,p,u) | \leq \beta( | \xi(0) - p |, u)$}.
\end{enumerate}
\end{definition}
Explicit growth bounds of the form 
\begin{IEEEeqnarray}{c}\label{e:growthbound}
\beta(r,u)=\e^{L(u) \tau} r + v(u),
\end{IEEEeqnarray}
where $v(u) \in \mathbb{R}_{+}^n$ and the matrix
$L(u) \in \mathbb{R}^{n\times n}$ is  essentially nonnegative, can be
computed under mild assumptions
\cite{i14sym}.
The next result, which extends \cite[Th.~VIII.4]{i14sym} to the case
of multiple growth bounds, is the key to the computation of abstractions.
\begin{theorem}
\label{t:abstraction}
Let $S_1=(X_1,U_1,F_1)$ be the sampled system associated with~\ref{e:System:c-time}
and sampling time $\tau>0$, and let
$P$ be given
by~\ref{e:cSys:meas}.
Let $S_2=(X_2,U_2,F_2)$ be a system, where 
$X_2$ is a cover of $X_1$ by nonempty, closed hyper-intervals and
$U_2\subseteq U_1$. Consider a subset $\bar X_2\subseteq X_2$ that
satisfies~\ref{e:abs:ss}, and for any $x_2 \in \bar X_2$
let $\beta_{x_2}$ be a growth bound on $P(x_2)$, 
$U_2$ associated with $\tau$ and \ref{e:System:c-time}.
Suppose that $F_2$ is given by
\begin{enumerate}
  \item \label{d:abstraction:rhs:1}
  $F_2(x_2,u)=\emptyset$ whenever $x_2\in X_2\setminus \bar X_2$, $u \in U_2$, and
  \item \label{d:abstraction:rhs:2}
  for $x_2\in \bar X_2$, $x'_2\in X_2$ and $u\in U_2$ we have
  \begin{IEEEeqnarray}{c}\label{e:abs:tf}
   x'_2\in F_2(x_2,u) \iff  \left( c + \segcc{-r',r'} \right) \cap
    P(x'_2)\neq\emptyset,
    \IEEEeqnarraynumspace
  \end{IEEEeqnarray}
with $r'=\beta_{x_2}(\eta/2+z,u)$, $x_2=\bar c+\segcc{-\eta/2,\eta/2}$, and
\begin{equation}
\label{e:t:abstraction}
c = \varphi(\tau,\bar c , u).
\end{equation}
\end{enumerate}
Then we have $S_1\preccurlyeq_{Q\circ P}S_2$, with $Q\subseteq
X_1\times X_2$ defined by
$(x_1,x_2)\in Q$ iff $x_1\in x_2$.
\end{theorem}
Theorem~\ref{t:abstraction} leads to constructive means to compute abstractions basically as follows. 
For every cell $x_2=\bar c+\segcc{-\eta/2,\eta/2}\in \bar X_2$ and input symbol $u\in U_2$
\begin{enumerate}[1)]
\item compute $c=\varphi(\tau,\bar c,u)$ and $r' = \beta_{x_2}(\eta/2+z,u)$,
\item determine all cells $c'+\segcc{-\eta/2,\eta/2} \in X_2$ that satisfy
  \begin{IEEEeqnarray*}{c}
    \left( c + \segcc{-r',r'} \right) \cap
  \left (  c'+\segcc{-\eta/2-z,\eta/2+z} \right )\neq\emptyset,
  \end{IEEEeqnarray*}
  and define $F_2(x_2,u)$ as the set of all such cells.
\end{enumerate}
\begin{proof}[Proof of Theorem \ref{t:abstraction}]
We have to verify \ref{d:FRR:in},\ref{d:FRR:dyn} in 
Definition \ref{d:FRR} with $Q\circ P$ in place of $Q$. To see \ref{d:FRR:in}
let $(x_1,x_2) \in Q \circ P$ and $u \in U_{S_2}(x_2)$. 
Then, $F_2(x_2,u) \neq \emptyset$
and $x_2 \in \bar X_2$ by the assumptions on $u$ and $F_2$. 
By our assumption on the 
growth bound $\beta_{x_2}$ it follows $F_1(x_1,u) \neq \emptyset$, thus $u \in U_{S_1}(x_1)$.\\
To see \ref{d:FRR:dyn} in Definition \ref{d:FRR}, 
let $(x_1,x_2) \in Q \circ P$, 
$u \in U_{S_2}(x_2)$ and
$x_2' \in (Q\circ P)(F_1(x_1,u))$. 
It follows that 
$F_1(x_1,u) \cap P(x_2') \neq \emptyset$. Indeed,
$x_2' \in (Q \circ P)(x_1')$ for some $x_1' \in F_1(x_1,u)$, thus
$(x_1',P(x_2')) \in Q$, so $x_1' \in P(x_2')$. 
Next, from $x_1 \in P(x_2)$ and the properties of $\beta_{x_2}$ 
it follows that $F_1(x_1,u) \subseteq \varphi(\tau,\bar c,u) + \segcc{-r',r'}$, where
$x_2 = \bar c + \segcc{-\eta/2,\eta/2}$ and $r' = \beta_{x_2}(\eta/2+z,u)$. Thus, 
$(\varphi(\tau,\bar c,u) + \segcc{-r',r'})\cap P(x_2') \neq \emptyset$, 
and by the properties of $F_2$, we conclude $x_2' \in F_2(x_2,u)$. 
\end{proof}

\section{Estimation of the size of abstractions}
\label{s:estimation}
The size of an abstraction $S_2 = (X_2,U_2,F_2)$ that is obtained by 
Theorem \ref{t:abstraction} is given by the number of transitions.
To obtain a \emph{prediction} on this size, 
we will disregard overflow symbols 
by assuming $\bar X_2 = X_2$, 
and in addition, we will assume $c$ in \ref{e:abs:tf} 
is a random vector uniformly distributed on the cells. Then, 
the following theorem shows that 
the function $E \colon (\mathbb{R}_+ \setminus \{0\})^n \to \mathbb{R}_+$ 
given by \ref{e:DLFunction}
with $A \in \mathbb{R}^{n \times n}_+$, $p \in \mathbb{R}^n_+$ 
provides a prediction on the cardinality of $F_2(x_2,u)$
for fixed $(x_2,u) \in \bar X_2 \times U_2$ 
in dependence of the grid parameter $\eta$.

The key property of the functional $E$ is that 
it also provides 
an accurate prediction when actually computing abstractions.
(See Section \ref{s:example}.)

\begin{theorem}
\label{t:size}
Assume the hypotheses of Theorem \ref{t:abstraction} with 
$\bar X_2 = X_2$. 
Let $(x_2,u) \in X_2 \times U_2$ and let $\beta \defas \beta_{x_2}$ in
Theorem \ref{t:abstraction} be of the form \ref{e:growthbound}, where
$v(u) \ge 0$ and $L(u)$ is essentially nonnegative.
For $c$ in \ref{e:abs:tf} assume in place of \ref{e:t:abstraction}
that $c$ is an $n$-dimensional vector of independent random variables
$c_i$, $i \in \intcc{1;n}$ each of which is uniformly distributed on
some interval of length $\eta_i$. Then the expected value of the
number of cells in $F_2(x_2,u)$ is given by 
$E(\eta)$ in \ref{e:DLFunction} with
\begin{IEEEeqnarray}{c't'c}
\label{e:ApinE}
A = \id + \e^{L(u) \tau} & \text{and} & p = 2(A z + v(u) ).%
\end{IEEEeqnarray}
\end{theorem}
\begin{proof}
The number of elements in $F_2(x,u)$ for $x = \bar c + \segcc{-\eta/2,\eta/2}$ is given by
\begin{IEEEeqnarray}{c}\label{e:num:cells}
\begin{IEEEeqnarraybox}[][c]{l}
N_\eta^n(c,r')=
|\{p\in \eta\Z^n \mid p\in
c+\segcc{-r',r'}\}|
\IEEEeqnarraynumspace
\end{IEEEeqnarraybox}
\end{IEEEeqnarray}
with the random vector $c$ and \mbox{$r'=r+\e^{L(u) \tau} r + v(u)$} where $r = \eta/2+z$.
Here, $|X|$ stands for the cardinality of the set $X$.
We have $2r' = 2 r + 2 (\e^{L(u) \tau} r + v(u) ) = p + A \eta$.
The proof is therefore completed by the next lemma.
\end{proof}
\begin{lemma}
Consider the grid $\eta\Z^n$ in \ref{e:grid} with $\eta\in\R^n$, $\eta>0$. Let
$r\in\R_+^n$ and let 
$c_i$, $i\in\intcc{1;n}$ be $n$ independent random variables, where each
$c_i$ is uniformly distributed on some interval of length $\eta_i$.
Then the expected value of the number $N_\eta^n(c,r)$ defined in \ref{e:num:cells} 
is given by $\prod_{i=1}^n {2r_i}/{\eta_i}.$
\end{lemma}
\begin{proof}
Note that $N_\eta^n(c,r)=\prod_{i=1}^n N_{\eta_i}^1(c_i,r_i)$ and
since the $c_i$ are mutually independent, the expected value of
$N_\eta^n(c,r)$ is given as the product of the expected values of
$N^1_{\eta_i}(c_i,r_i)$. Moreover, 
$N^1_{\eta_i}(c_i,r_i) = N^1_{1}(c_i/\eta_i,r_i/\eta_i)$
and $N_1^1(x + 1, \hat \eta) = N_1^1(x,\hat{\eta})$ for every 
$x \in \mathbb{R}$ and $\hat \eta \in \mathbb{R}_+$.
Hence, it suffices to consider $N_1^1(\hat c,\hat \eta)$ with $\hat c$ being
uniformly distributed on $\intcc{0,1}$ and $\hat \eta\in\R_+$.
The expected value of $N_1^1(\hat c , \hat \eta)$ is given by
$2\hat \eta$. Indeed, if $\hat \eta = k + \varepsilon$ with
$k \in \mathbb{Z}_+$ and $\varepsilon \in \intco{0,1/2}$, then we obtain
\begin{equation*}
\int_{0}^1 N_1^1(x,\hat \eta) \ \mathrm{d}x
=
(2k+1) \varepsilon
+
2k(1-2\varepsilon)
+
(2k+1)\varepsilon
=
2 \hat \eta
\end{equation*}
by separating the integration interval into
$\intcc{0,\varepsilon}$, $\intcc{\varepsilon,1-\varepsilon}$ and
$\intcc{1-\varepsilon,1}$.
The case $\varepsilon \in \intco{1/2,1}$ is similar.
\end{proof}
\section{Minimization of the size of abstractions}
\label{s:op}
Theorem \ref{t:size} motivates the following on the computation of abstractions in the special case that
the growth bounds in Theorem \ref{t:abstraction} coincide and do not depend on the input symbol, i.e., $\beta_{x_2}(r,u) = \beta_{x_2'}(r,u')$ for any $(x_2,u),(x_2',u') \in \bar X_2 \times U_2$, and any $r \in \R^n_+$: Consider the abstractions for $S_1$ that have cells of volume $\exp(\gamma)$, $\gamma \in \R$, and input alphabet $U_2$. Among those abstractions, the abstraction with the least expected size has cells that are aligned according grid parameter $\eta$, where $\eta$ is a solution of the optimization problem
\begin{equation}
\label{e:E:op}
\min_{\xi > 0} E(\xi) \text{ subject to } 
\exp(\gamma) = \prod_{i=1}^n\nolimits \xi_i. 
\end{equation}
Unfortunately, the optimization problem \ref{e:E:op} is non-convex if
$n \ge 2$, and non-convex problems are notoriously difficult to solve.

The main results of this work, which are presented in this section,
include a characterization of existence and uniqueness of $\eta$, and
the means to numerically compute $\eta$, so that the just motivated
heuristic to reduce the computational effort becomes applicable. We
will also investigate the generalization of \ref{e:E:op} to the case
of arbitrary growth bounds.

To establish aforementioned characterization we first bypass non-convexity. To this end, consider a transformation of \ref{e:E:op}:
\begin{equation}
\label{e:g:op}
\min_x g(x) \text{ subject to } x \in V_\gamma,
\end{equation}
where $g(x) = E(\exp(x))$ and
\begin{equation}
\label{e:V}
V_s  = 
\{ v \in \mathbb{R}^n \mid v_1+\ldots + v_n = s\}
\end{equation}
for $s \in \R$. Here and subsequently, the exponential $\exp$ is taken component-wise whenever the argument is a vector. The result below lists the outstanding properties of $g$.

\begin{theorem}
\label{t:g}
Let $n \ge 2$, $A \in \mathbb{R}_{+}^{n \times n}$, $p \in \mathbb{R}_{+}^n$ and
$\gamma \in \mathbb{R}$, and let $E$, $g$ and $V$ be defined as in
\ref{e:DLFunction} and above. Then $g$ is convex. Moreover, if all diagonal entries of $A$ are positive and 
$A$ or 
$\left ( \begin{smallmatrix} A & p \\ \mathbf{1} & 1 \end{smallmatrix} \right )$ 
is irreducible then the assertions below hold.
\begin{enumerate}[(i)]
\item 
\label{t:g:strictlyconvex}
$g$ is strictly convex on $V_\gamma$. 
To be more precise, \mbox{$g''(x)h^2>0$} for all $x \in V_\gamma$ and $h \in V_0 \setminus \{0\}$.
\item 
\label{t:g:2}
Let $\mu$ be the smallest nonzero entry of 
$A$ and $p$, \mbox{$c=(n-1)^{-1}$}. 
Then $x \in V_\gamma$ implies
\begin{equation}
\label{e:t:g:lowerbound}
g(x) \geq \mu^n \exp(-|\gamma|c) \exp( c \|x\|_\infty ).
\end{equation}
\end{enumerate}
\end{theorem}

The above result implies that the optimization problem \ref{e:g:op} is
convex. Moreover, since \ref{e:t:g:lowerbound} implies
$g(x) \to \infty$ as \mbox{$\|x\|_\infty \to \infty$}, $x \in V_\gamma$,
the problem \ref{e:g:op} has a unique solution under the hypotheses of
Theorem \ref{t:g} \cite[4.3.3]{OrtegaRheinboldt00}, and thus, so has
\ref{e:E:op}.
Our result also shows that standard numerical methods will
converge globally when applied to \ref{e:g:op},
e.g.~\cite[Sec.~14.5]{OrtegaRheinboldt00},
and some will do so even if \ref{e:g:op} is supplemented with a finite
number of constraints of the form
$a_i \leq x_i$ or $x_i \leq a_i$,
$a_i \in \R$ \cite[Th.~1]{Solodov09}.
For completeness, we remark that global convergence can also
be ensured if $g$ is strongly convex on $V_\gamma$, a property
established in \cite[Th.~3]{i15gridc} for irreducible $A$. However,
the property is not implied under the rather
mild hypotheses of Theorem \ref{t:g}, which are satisfied, in
particular, if every component of the state is subject to some
measurement error, i.e., if $z > 0$ in \ref{e:ApinE}, regardless of
the dynamics of the plant $S_1$ under investigation. Finally,
 we note that 
checking irreducibility (for $n \geq 2$) is equivalent to 
finding strongly connected components 
in directed graphs \cite[Th.~2.2.7]{BermanPlemmons94}, so
irreducibility can be checked with linear time algorithms \cite{Tarjan72}.
\begin{proof}[Proof of Theorem \ref{t:g}]
Convexity of $g$ has been established in \cite[Th.~3]{i15gridc}. We
first prove \ref{t:g:strictlyconvex}. Define
$
R_i(x)
=
p_i + \sum_{j=1}^n A_{i,j} \e^{x_j}
$
and
$
G(x)
=
\sum_{i=1}^n
\ln R_i(x)
$
to see that
\[
g(x)
=
\e^{-\gamma} \prod_{i=1}^n R_i(x)
=
\exp( G(x) - \gamma )
\]
for every $x \in V_{\gamma}$.
It follows that
$
g''(x)h^2
=
g(x)\left( G'(x)h \right)^2
+
g(x) G''(x)h^2
$
and that
$
(G'(x)h)^2
+
G''(x)h^2
$ 
equals
\begin{equation}
\label{e:SumInR}
\left(
\sum_{i=1}^n \frac{R_i'(x)h}{R_i(x)}
\right)^2
+
\sum_{i=1}^n
\frac{R_i''(x)h^2 R_i(x) - (R_i'(x)h)^2}{(R_i(x))^2}
\end{equation}
for all $x \in V_{\gamma}$ and all $h \in V_0$.
Define vectors $a^{(i)}, b^{(i)} \in \mathbb{R}^{n+1}$ by
$b^{(i)}_{n+1} = p_i^{1/2}$, $a^{(i)}_{n+1} = 0$,
$b^{(i)}_j
=
\left( A_{i,j} \e^{x_j} \right)^{1/2}
$
and
$
a^{(i)}_j
=
b^{(i)}_j h_j
$
for every $j \leq n$. Then $b^{(i)}$ is not a zero vector 
for any $i \in \intcc{1;n}$. 
Use \ref{e:SumInR} to see that $g(x)^{-1}g''(x)h^2$ equals
\begin{IEEEeqnarray}{l}\label{e:ab}
\Big  ( \sum_{i=1}^n \frac{\innerProd{a^{(i)}}{b^{(i)}}}{\|b^{(i)}\|_2^2} \Big )^2 
+ 
\sum_{i=1}^n \frac{\|a^{(i)}\|_2^2\! \cdot\! \| b^{(i)}\|^2_2-\innerProd{a^{(i)}}{b^{(i)}}^2}{\|b^{(i)}\|_2^4}.
\IEEEeqnarraynumspace
\end{IEEEeqnarray}
Now assume $g''(x)h^2 = 0$ for some
$h \in V_0$ and let us show that \mbox{$h=0$}. 
Indeed, we deduce from \ref{e:ab} and Cauchy's inequality 
that for all $i \in \intcc{1;n}$ there exists $\lambda_i \in \R$ such that 
$a^{(i)} = \lambda_i b^{(i)}.$
This equation implies 
a) $\lambda_i = 0$ 
whenever $p_i>0$, 
b) $\lambda_i = h_i$ for any $i$ as $A_{i,i}>0$, and therefore 
c) $h_i = h_j$ whenever $A_{i,j}>0$.  
Next, assume that $h$ has two nonzero components $h_r$, $h_s$ such that $h_r \neq h_s$. 
By the irreducibility of $A$ or $\left ( \begin{smallmatrix} A & p \\ \mathbf{1} & 1 \end{smallmatrix} \right )$ 
there exist distinct indices $i_1,\ldots,i_m \in \intcc{1;n+1}$ such that $r = i_1$, $s = i_m$ 
and at least one of the following cases occurs:
\begin{enumerate}[1)]
\item $n+1 \notin \{i_1,\ldots,i_m\}$ and $A_{i_k,i_{k+1}} > 0$ for all $k\leq m-1$,
\item $p_{r} > 0$,
\item $m\geq 3$, $p_{i_{m-1}}> 0$ and $A_{i_k,i_{k+1}} > 0$ for all $k \leq m-2$.
\end{enumerate}
The remarks a), b), c) above will exclude each of the three cases.
Indeed, the first case is impossible as it implies $h_r = h_s$. 
The second case implies $h_r = 0$, so cannot occur either. 
For the same reason, the third case is impossible as $h_{i_{m-1}} = 0$ and $h_r = h_{i_{m-1}}$. 
Consequently, the nonzero entries of $h$ coincide.
However, as \ref{e:ab} vanishes, we conclude
$0 = \sum_{i=1}^n \lambda_i = \sum_{i=1}^n h_i$, 
and so $h = 0$.\\
Now we prove \ref{t:g:2}.
We begin with deriving an inequality that we use in the second part
of the proof.
To this end, note first that by our assumptions
the following property holds for all $r \in \intcc{1;n}$ and 
all $s \in \intcc{1;n}$, or for all $r \in \intcc{1;n}$ and $s = n+1$:
There exists a subset of indices $\{i_1,\ldots,i_m\} \subseteq \intcc{1;n+1}$ such that 
$(r,s) = (i_1,i_m)$ and $\tilde A_{i_k,i_{k+1}} > 0$ for all $k\in \intco{1;m}$, where 
$\tilde A = \left ( \begin{smallmatrix} A & p \\ 0 & 1 \end{smallmatrix} \right )$. Now fix $(r,s)$ 
satisfying previous 
condition with the subset of indices $P \defas \{i_1,\ldots,i_m\}$. 
Let $x \in \R^n$ and set $x_{n+1} \defas 0$ to see that
\begin{IEEEeqnarray}{cc}
&g(x)\geq \notag\e^{-\gamma}
\Big(
\prod_{j=1 \atop j \notin P\setminus \{i_m\}}^n
\tilde A_{j,j} \exp(x_j)
\Big)
\prod_{k=1}^{m-1}
\tilde A_{i_k, i_{k+1}}
\exp(x_{i_{k+1}})\\
&
\geq
\mu^n
\Big(
\prod_{j \in P} \exp(-x_j)
\Big)
\prod_{k=2}^m \exp(x_{i_k}) =\mu^n \exp(x_{s} - x_{r}).
\label{e:basicineq} 
\IEEEeqnarraynumspace
\end{IEEEeqnarray}
For the second part of the proof let $q\in \intcc{1;n}$ such that $\|x\|_\infty = |x_q|$ and 
observe $x_q - \gamma = - \sum_{r=1,r \neq q}^n x_r$.
So, for some $s \in \{q, n+1\}$ we have
\begin{equation}
\label{e:prod:exp}
\prod_{r=1 \atop r \neq q}^n \exp(x_s-x_r) = \exp((n-1)x_s)\exp(x_q-\gamma).
\end{equation}
Therefore,
$( \mu^{-n} g(x))^{n-1}$ is not less than \ref{e:prod:exp} by applying \ref{e:basicineq} 
to each factor of the product in \ref{e:prod:exp}. This implies \ref{e:t:g:lowerbound}
if $x_q\geq 0$. If $x_q<0$ and if \ref{e:basicineq} holds for $s = n+1$ then \ref{e:t:g:lowerbound} follows 
obviously. 
If $x_q<0$ and if \ref{e:basicineq} holds for all $r,s \in \intcc{1;n}$ then
take the inverse on both sides of \ref{e:prod:exp} with $s=q$, and use
\ref{e:basicineq} with $r$ and $q$ in place of $s$ and $r$,
respectively, to see that $(\mu^{-n}g(x))^{n-1}\geq \exp(\gamma - n x_q)$. 
So, the proof is easily completed. %
\end{proof}

The announced characterization related to \ref{e:E:op} is as follows.

\begin{theorem}
\label{t:E}
Let $\gamma \in \R$ and
$E$ be defined as in \ref{e:DLFunction} with $A$ having positive diagonal. 
Then the optimization problem \ref{e:E:op}
has a unique minimum point 
iff
$A$ or 
$\left ( \begin{smallmatrix} A & p \\ \mathbf{1} & 1 \end{smallmatrix}
\right )$ is irreducible.
\end{theorem}

The proof of above theorem requires the next lemma.
Below, $A^\top$ denotes the transpose of a matrix $A$.
\begin{lemma}
\label{l:structure}
Let $n\ge 2$, let $A \in \mathbb{R}^{n\times n}$, $p \in \mathbb{R}^n_+$. Assume that both 
$A$ and 
$\tilde A = \left ( \begin{smallmatrix} A & p \\ \mathbf{1} & 1 \end{smallmatrix} \right )$ are reducible. 
Then $\tilde P\tilde A \tilde P^\top$ equals 
\begin{equation}
\label{e:PAPT}
\begin{pmatrix}
X & 0 & 0 \\ 
Z & Y & \bar p \\
\mathbf{1} & \mathbf{1} & 1 
\end{pmatrix}
\end{equation}
where $\tilde P = (\begin{smallmatrix}
P & 0 \\ 0 & 1
\end{smallmatrix} ) \in \R^{(n+1) \times (n+1)}$,
$P \in \mathbb{R}^{n\times n}$ is a permutation matrix,
$X \in \R^{n_1 \times n_1}$,
$Y \in \R^{n_2 \times n_2}$,
$Z \in \R^{n_2 \times n_1}$,
$n_1,n_2 \in \intco{1;n}$,
$n_1+n_2 = n$ and
$\bar p \in \R^{n_2}_+$.
\end{lemma}
\begin{proof}
There exists a permutation matrix $T \in \R^{(n+1) \times (n+1)}$
such that
$T\tilde A T^\top 
= 
\left ( \begin{smallmatrix}
B & 0 \\
D & C 
\end{smallmatrix} \right )
$,
where 
$B \in \mathbb{R}^{m_1 \times m_1}$, 
$C \in \mathbb{R}^{m_2 \times m_2}$,
$D \in \mathbb{R}^{m_2\times m_1}$,
$m_1,m_2 \in \intcc{1;n}$ and 
$m_1+m_2 = n+1$. This fact may be seen by using \cite[Th.~2.2.7]{BermanPlemmons94} to
establish the equivalence of the definition of irreducibility in Section \ref{s:prelims} 
and \cite[Def. 2.1.2]{BermanPlemmons94} for $n \geq 2$.
Next, as the last row of $\tilde A$ contains only nonzero entries, 
we conclude that the nonzero entries of $p$ are contained in a column of 
$C$.
Hence, we may assume without loss of generality that
the nonzero entries of $p$ are contained in the last column of
$C$,
i.e., that
$T = (\begin{smallmatrix}
P & 0 \\ 0 & 1
\end{smallmatrix} )$, 
where $P \in \R^{n \times n}$ is a permutation matrix. 
Then, observe that the upper left $n\times n$ submatrix of $T\tilde A T^\top$ equals 
$PAP^\top$ and the first $m_1$ entries
of the last column of $T\tilde A T^\top$ vanish.
So, as $A$ is reducible, we may redefine $P$ to be such that $P A P^\top =(\begin{smallmatrix}
X & 0 \\ Z & Y
\end{smallmatrix} )
$, which completes the proof.
\end{proof}

Below, we denote $E$ by $E_{A,p}$ whenever clarity requires 
to specify
$A \in \R^{n \times n}_+$ and $p \in \R^n_+$ in the definition
\ref{e:DLFunction} of $E$.

\begin{proof}[Proof of Theorem \ref{t:E}]
Sufficiency has already been established in the
remark following the statement of Theorem \ref{t:g}.
To prove necessity, assume that both matrices in the statement are reducible.
Note that $E_{A,p}(x) = E_{PAP^\top,Pp}(Px)$ for any permutation matrix $P \in \R^{n\times n}$ 
and $x \in (\R_+ \setminus \{0\})^n$. Therefore, assume without loss of generality that 
$\left ( \begin{smallmatrix} A & p \\ \mathbf{1} & 1 \end{smallmatrix} \right )$ 
is of the form \ref{e:PAPT}. For $x \in \R^n$ let $x^{(1)} = (x_1,\ldots,x_{n_1})$ 
and $x^{(2)}=(x_{n_1+1},\ldots,x_{n})$, i.e., $x =
(x^{(1)},x^{(2)})$. Then
\begin{equation} 
\label{e:product}
 E(x) = \prod_{i=1}^{n_1} (Xx^{(1)})_i \prod_{i=1}^{n_2} (Zx^{(1)} + Yx^{(2)} + \bar p)_i
\end{equation}
for any $x \in (\R_+ \setminus \{0\})^n$. Now, assume $x$ is a solution of \ref{e:E:op} and set
$\xi_{\lambda} = (\lambda^{-n_2/n_1}x^{(1)},\lambda x^{(2)})$
for every $\lambda>0$. 
It follows that $\xi_{\lambda}$ is a feasible point of \ref{e:E:op},
for all $\lambda>0$.
Next, if $Zx^{(1)} \neq 0$ or $\bar p \neq 0$
we obtain using \ref{e:product} that
\begin{equation}
\label{e:lowerboundonf}
E(x) > E_{X,0}(x^{(1)})\cdot E_{Y,0}(x^{(2)}).
\end{equation}
However, $E(\xi_{\lambda})$ converges to 
the right hand side of \ref{e:lowerboundonf} as $\lambda \to
\infty$, which contradicts the choice of $x$. 
If $Zx^{(1)} = 0$ and $\bar p = 0$ then
$E(x) = E(\xi_\lambda)$ for any $\lambda$, so any $\xi_\lambda$ is
a solution of \ref{e:E:op} for any $\lambda > 0$.
\end{proof}

Finally, we consider the general situation in Theorems
\ref{t:abstraction} and \ref{t:size}, where the growth bounds depend
on the cell and the input symbol.
In this case, the computation of an abstraction with finite $\bar X_2$ and finite $U_2$ requires a sequence of growth bounds $(\beta_j)_{j \in J}$ indexed by some finite set $J$. 
In particular, for any $i=(j,u) \in J \times U_2$ there are
an essentially nonnegative matrix $L^{(i)} \in \R^{n\times n}$ and $v^{(i)} \in \R^n_+$ such that $\beta_j(r,u)$ is given by the right hand side of \ref{e:growthbound} with 
$L^{(i)}$ and $v^{(i)}$ in place of
$L(u)$ and $v(u)$, respectively, for all $r \in \R^n_+$.
Therefore, we obtain by generalizing \ref{e:E:op} the following heuristic to reduce the size of abstractions that have cells of volume $\exp(\gamma)$, $\gamma\in \R$: Pick the grid parameter to solve
\begin{equation}
\label{e:heuristic}
\min_{\xi > 0} \widetilde{E}(\xi) \text{ subject to } 
\exp(\gamma) = \prod_{i=1}^n\nolimits \xi_i,
\end{equation}
where $\widetilde{E} \colon (\R_+ \setminus \{0\})^n \to \R_+$ is given by
\begin{equation}
\label{e:Estar}
\widetilde{E}(\xi) = \sum_{i\in I}\nolimits E_{A^{(i)},p^{(i)}}(\xi).
\end{equation}
Here, $A^{(i)}$ and $p^{(i)}$ are defined analogously to $A$ and $p$ in \ref{e:ApinE}, i.e.,
$A^{(i)}=\id + \exp(L^{(i)}\tau)$, 
$p^{(i)}=2(A^{(i)}z+ v^{(i)})$.

Theorem \ref{t:optimaleta} below provides sufficient conditions for \ref{e:heuristic} to 
possess a unique solution. To facilitate the practical verification of the conditions we will formulate them in terms of $L^{(i)}$ and $v^{(i)}$ rather than in terms of $A^{(i)}$ and $p^{(i)}$.
Moreover, Theorem \ref{t:optimaleta} will imply that \ref{e:heuristic} can be solved numerically by solving \ref{e:g:op} with $\widetilde{g}(x)\defas\widetilde{E}(\exp(x))$ in place of $g(x)$.

\begin{theorem}
\label{t:optimaleta}
Let 
$\gamma \in \R$,
$\tau > 0$,
$z \in \R^n_+$, 
and let
$I$ be a finite set.
For every $i \in I$ define
$A^{(i)}=\id + \exp(L^{(i)}\tau)$, 
$p^{(i)}=2(A^{(i)}z+ v^{(i)})$, 
where
$L^{(i)} \in \R^{n\times n}$ is essentially nonnegative, and 
$v^{(i)} \in \R^n_+$.
Let $V$, $\widetilde{E}$ and $\widetilde{g}$ be defined as in \ref{e:V}, \ref{e:Estar} and above.
For some $i\in I$ suppose that 
$L^{(i)}$ or
\begin{equation}
\label{e:t:optimaleta:tildeL}
\left ( 
\begin{matrix}
L^{(i)} & z + L^{(i)}z + v^{(i)} \\ 
\mathbf{1} & 1
\end{matrix} 
\right )
\end{equation}
is irreducible.
Then \ref{e:heuristic} possesses a unique solution.
Moreover, $\widetilde{g}\,''(x)h^2 > 0$ for all $x \in V_\gamma$ and $h \in V_0 \setminus \{0\}$, and
$\widetilde{g}(x)\to \infty$ as $\|x\|_\infty \to \infty$, $x \in V_{\gamma}$.
\end{theorem}

The proof of Theorem \ref{t:optimaleta} requires the result below.

\begin{lemma}
\label{l:expLt}
Let $p \in \mathbb{R}^n_+$ and let $L \in \mathbb{R}^{n \times n}$ be essentially nonnegative. 
We have the following:
\begin{enumerate}
\item
\label{l:expLt:L}
$L$ is irreducible iff $\exp(Lt)$ is irreducible for every $t>0$.
\item 
\label{l:expLt:tildeL}
If $L$ is reducible, then 
$\left ( \begin{smallmatrix} L & p \\ \mathbf{1} & 1 \end{smallmatrix} \right )$ 
is irreducible iff $\left ( \begin{smallmatrix} \exp(Lt) & p \\ \mathbf{1} & 1 \end{smallmatrix} \right )$ is irreducible for every $t>0$.
\end{enumerate}
\end{lemma}
\begin{proof}
Necessity in \ref{l:expLt:L} follows from
\cite[Rem.~I.7.9]{Kato82} and \cite[Th.~2.2.7]{BermanPlemmons94} 
by establishing the equivalence between the definition of irreducibility in 
\cite[Sec.~I.7.4]{Kato82} and ours.
Sufficiency in \ref{l:expLt:L} follows from
an evaluation of the exponential series.

To verify the necessary condition in \ref{l:expLt:tildeL}, first note that
\begin{equation}
\label{e:l:expLt:1}
(L_{i,j} \neq 0, \ t>0) \implies \exp({Lt})_{i,j} \geq 0
\end{equation}
if $L$ is nonnegative. To see that \ref{e:l:expLt:1} also holds for
essentially nonnegative matrices, choose $c>0$ such that $L + c \id$
is nonnegative. Then \ref{e:l:expLt:1} holds since
$\exp(Lt) = \exp((L + c \id)t) \exp(-ct)$.
Thus, $A=\left ( \begin{smallmatrix} \exp(Lt) & p \\ \mathbf{1} &
    1 \end{smallmatrix} \right )$ arises from $B = \left
  ( \begin{smallmatrix} Lt & p \\ \mathbf{1} & 1 \end{smallmatrix}
\right )$ by adding nonnegative values to the entries of $B$. Consequently, the irreducibility of $B$ passes over to $A$.
\\
To prove sufficiency, assume that both
$L$ and 
$\tilde L = \left ( \begin{smallmatrix} L & p \\ \mathbf{1} & 1 \end{smallmatrix} \right )$
are reducible. 
By Lemma \ref{l:structure} there exist permutation matrices $P$ and $\tilde P$
satisfying the following: $\tilde P \tilde L \tilde P^\top$ is of the form \ref{e:PAPT} and 
$P\exp(L)P^\top = \exp(PLP^\top) = \left ( \begin{smallmatrix} X' & 0 \\ Z' & Y' \end{smallmatrix} \right )$, 
where $X'$, $Y'$, $Z'$ are matrices of the same dimensions as $X$, $Y$, $Z$ in \ref{e:PAPT}, respectively. 
Hence, 
$
\left ( \begin{smallmatrix} \exp(L) & p \\ \mathbf{1} & 1 \end{smallmatrix} \right )
$ 
is reducible by \cite[Th.~2.2.7]{BermanPlemmons94} as it
can be transformed by $\tilde P$ to the form \ref{e:PAPT} with $X'$, $Y'$, $Z'$ in place of $X$, $Y$, $Z$, respectively.
\end{proof}

\begin{proof}[Proof of Theorem \ref{t:optimaleta}]
As \ref{e:l:expLt:1} holds for $L^{(i)}$ in place of $L$,
we may prove the theorem assuming $\exp(L^{(i)}\tau)z$ in place of $L^{(i)}z$ in \ref{e:t:optimaleta:tildeL}.
Thus, by Lemma \ref{l:expLt} 
one of the matrices
in the statement of Theorem \ref{t:g} with 
$A^{(i)}$, $p^{(i)}$ in place of $A$, $p$, respectively, 
is irreducible.
Moreover, $A^{(i)}$ has positive diagonal \cite[Th.~I.7.4]{Kato82}. 
Now apply Theorem \ref{t:g} to $A^{(i)}$ and $p^{(i)}$ in place of $A$ and $p$ to see that $a(x)\defas E_{A^{(i)}, p^{(i)}}(\exp(x))$ satisfies $a''(x)h^2>0$ for all $x \in V_\gamma$ and $h \in V_0 \setminus \{0\}$, and $a(x)\to \infty$ as $\|x\|_\infty \to \infty$, $x \in V_\gamma$. 
Moreover, every summand in $\widetilde{g}$ is convex
and positive. So, as $a$ is a summand in $\widetilde{g}$ the assertions on
$\widetilde{g}$ hold. Hence,
\ref{e:g:op} with $\widetilde{g}$ in place of $g$ possesses a unique solution
\cite[4.3.3]{OrtegaRheinboldt00}, and thus, so does \ref{e:heuristic}.\looseness=-1
\end{proof}
\section{Numerical example}
\label{s:example}
To demonstrate the benefits of the presented results, 
we consider the control system of a double pendulum that is mounted on a cart
as investigated in \cite{GraichenTreuerZeitz07}. The dynamics of this system
can be decomposed into the motion of the two poles and the motion of the cart,
which are coupled by the acceleration of the cart. Here, we consider only 
the motion of the poles. 
Specifically, 
we consider the equations of motion given in
\cite[Tab.~2]{GraichenTreuerZeitz07}, rewritten as a first order
system \ref{e:System:c-time} with $n=4$, $\bar U = \R$, 
and $x = (\phi_1,\phi_2,\dot \phi_1,\dot \phi_2)$ in the 
notation of \cite{GraichenTreuerZeitz07}. 
Specifically,
$x_1$ and $x_2$ denotes the angle formed by the inner and outer, respectively,
pole and the vertical ray, and $x_3$ and $x_4$ denote
the corresponding angular velocities.
The control input $u$ is the acceleration of the cart. 
See also \cite[Fig.~1]{GraichenTreuerZeitz07}. 
We additionally model uncertainties in friction forces in the links
by virtue of $w=(0,0,0.018,0.028)$ in \ref{e:System:c-time}.

We aim at steering the state of the system from the stable equilibrium point $(\pi,\pi,0,0)$
to an ellipsoid centered at the lower unstable equilibrium point $x_0=(\pi,0,0,0)$ given by
\mbox{$\{ x \in \R^4 \mid (x-x_0)^\top V (x-x_0) \leq 1 \}$} where
\[
V = 
\left(
\begin{smallmatrix}
 0.247 & 0.153 & -0.023 & -0.026 \\
 0.153 & 0.106 & 0.026 & -0.023 \\
 -0.023 & 0.026 & 8.24 & 3.893 \\
 -0.026 & -0.023 & 3.893 & 1.922 \\
\end{smallmatrix}
\right).
\]
We assume measurement errors \ref{e:cSys:meas} with 
$z = (b,b,2b,2b)$, $b=2\pi/2^{14}$ which 
are motivated by $14$-bit quantized measurements of the angles.
Additionally, we require the state $x\in\R^4$ of the system not to leave 
$
\bar X = \intcc{\pi/2,\pi+0.1} \times \intcc{0,2\pi} \times \intcc{-5.7,5.7} \times \intcc{-5.7,5.7}
$
and we identify $x$ and $x+(0,2\pi k,0,0)$ for any $k \in \mathbb{Z}$. 
The latter means that we do not impose restrictions on the outer angle.

We shall solve this control task for
the sampled system $S_1$ associated with \ref{e:System:c-time} 
and sampling time $\tau = 0.01$
using the synthesis procedure outlined 
in Sections \ref{s:intro} and \ref{s:abstractions},
in which we focus on the computation of abstractions.
Using Theorem \ref{t:abstraction}, we will compute two abstractions
$S_2 = (X_2, U_2, F_2)$ and $S_2' = (X'_2, U'_2,F'_2)$ for $S_1$,
where $S_2$ is based on a naive choice of the grid parameter, and the
grid parameter for $S_2'$ is chosen using the results in Section \ref{s:op}. 

We begin with the details to $S_2$. 
We let $U_2$ consist of $5$ elements equally spaced on
$9.81\cdot\intcc{-3.5,3.5}$.
Next, we let $\bar X_2$ in Theorem \ref{t:abstraction} be a cover of $\bar X$
and let $X_2$ be a
grid of the form \ref{e:grid} with grid parameter
\begin{align*}
\eta & =
(\tfrac{\pi/2+0.1}{118},\tfrac{2\pi}{118},\tfrac{11.4}{118},\tfrac{11.4}{118})
\approx
(0.014,0.053,0.097,0.097).
\end{align*}
$\eta$ is a naive choice as 
each component of $\bar X$ is subdivided into $118$ intervals of equal length. 
Then $\bar X_2$ consists of about \mbox{$194\!\cdot\!10^6$} cells.
The transition function $F_2$ is computed according to Theorem \ref{t:abstraction}, 
where the required growth bounds are obtained by methods presented in \cite{i14sym}. 
The computation of $S_2$ requires $206$GB RAM and \mbox{$2.5$h} cpu time.
$S_2$ contains about $55.7\!\cdot\!10^9$ transitions and 
$54.6\!\cdot\!10^9$ transitions have been predicted by means of \ref{e:Estar}.
All computations in this section are run on a single thread of an Intel Xeon E5-2687W
($3.1$ GHz).

\begin{figure}
\centering
\includegraphics[width=\ifCLASSOPTIONonecolumn.49\fi\linewidth]{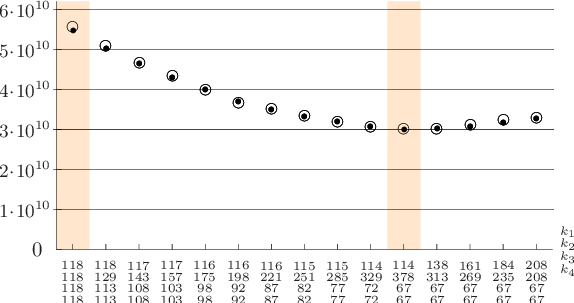}%
\caption{\label{fig:2Pendel}Predicted ($\bullet$) and actual ($\circ$)
  number of transitions in abstractions based on the grid parameters
  $({(\pi/2+0.1)}/{k_1},{2\pi}/{k_2},{11.4}/{k_3},{11.4}/{k_4})$ for the system $S_1$ considered in
  Section \ref{s:example}. 
}
\vspace*{-\baselineskip}%
\end{figure}

In contrast, for $S_2'$ we let $U_2'=U_2$ but choose the grid
parameter as the solution of \ref{e:heuristic} under the
constraint that the cells of $X_2'$ have the same volume as those of
$X_2$ and the
additional constraints $\xi_i \leq 0.17$, $i \in \intcc{1;4}$,
\begin{align*} 
\eta' & =
(\tfrac{\pi/2+0.1}{114},\tfrac{2\pi}{378},\tfrac{11.4}{67},\tfrac{11.4}{67})
\approx
(0.015,0.017,0.17,0.17)
\end{align*}
and define $X_2'$ in the same way as $X_2$, with $\eta '$ in place of $\eta$. 
The computation of $S_2'$ requires $134$GB RAM and 
$73$ min cpu time. $S_2'$ contains about $30.3\!\cdot\!10^9$ transitions 
and the prediction has been $30.0\!\cdot\!10^9$ transitions.

In summary, the number of transitions, computational time and memory
consumption is reduced by ${46\%}$, ${51\%}$ and
${35\%}$, respectively, compared to a naive choice of the
aspect ratio.
Moreover, in contrast to the auxiliary control problem for $S_2$, the
one for $S_2'$ is solvable, due to a reduced number of spurious transitions.

The success of our method to reduce the size of abstractions depends
to a great extend on the accuracy by which the functional in
\ref{e:heuristic} predicts the number of transitions.
That accuracy is illustrated in \ref{fig:2Pendel} for a number of
additional abstractions for $S_1$ with varying grid
parameters. It turns out that the prediction possesses an
error of less than $2\%$.

\bibliographystyle{IEEEtran}
\bibliography{GR/IEEEtranBSTCTL,GR/preambles,GR/mrabbrev,GR/strings,GR/fremde,GR/eigeneCONF,GR/eigeneJOURNALS,GR/eigenePATENT,GR/eigeneREPORTS,GR/eigeneTALKS,GR/eigeneTHESES,tmp}

\def\ocirc#1{\ifmmode\setbox0=\hbox{$#1$}\dimen0=\ht0 \advance\dimen0
  by1pt\rlap{\hbox to\wd0{\hss\raise\dimen0
  \hbox{\hskip.2em$\scriptscriptstyle\circ$}\hss}}#1\else {\accent"17 #1}\fi}
  \def\cprime{$'$} \ifx\hyperbaseurl\undefined
  \def\href#1#2{#2}\def\url#1{\texttt{#1}} \fi
  \ifx\ExplicitURLsInBibTeX\undefined\relax\else
  \def\href#1#2{www.reiszig.de/gunther/#1}\def\url#1{\texttt{#1}} \fi
\begin{thebibliography}{10}
\providecommand{\url}[1]{#1}
\csname url@samestyle\endcsname
\providecommand{\newblock}{\relax}
\providecommand{\bibinfo}[2]{#2}
\providecommand{\BIBentrySTDinterwordspacing}{\spaceskip=0pt\relax}
\providecommand{\BIBentryALTinterwordstretchfactor}{4}
\providecommand{\BIBentryALTinterwordspacing}{\spaceskip=\fontdimen2\font plus
\BIBentryALTinterwordstretchfactor\fontdimen3\font minus
  \fontdimen4\font\relax}
\providecommand{\BIBforeignlanguage}[2]{{%
\expandafter\ifx\csname l@#1\endcsname\relax
\typeout{** WARNING: IEEEtran.bst: No hyphenation pattern has been}%
\typeout{** loaded for the language `#1'. Using the pattern for}%
\typeout{** the default language instead.}%
\else
\language=\csname l@#1\endcsname
\fi
#2}}
\providecommand{\BIBdecl}{\relax}
\BIBdecl

\bibitem{i14sym}
G.~Reissig, A.~Weber, and M.~Rungger, ``Feedback refinement relations for the
  synthesis of symbolic controllers,'' \emph{IEEE Trans. Automat. Control},
  vol.~62, no.~4, pp. 1781--1796, Apr. 2017,
  \href{http://dx.doi.org/10.1109/TAC.2016.2593947}{DOI:10.1109/TAC.2016.2593947}, \href{http://arxiv.org/abs/1503.03715}{arXiv:1503.03715}.

\bibitem{Tabuada09}
P.~Tabuada, \emph{Verification and control of hybrid systems}.\hskip 1em plus
  0.5em minus 0.4em\relax New York: Springer, 2009.

\bibitem{RunggerStursberg12}
M.~Rungger and O.~Stursberg, ``On-the-fly model abstraction for controller
  synthesis,'' in \emph{American Control Conference (ACC)}, 2012, pp.
  2645--2650.

\bibitem{RunggerMazoTabuada13}
M.~Rungger, M.~Mazo, and P.~Tabuada, ``Specification-guided controller
  synthesis for linear systems and safe linear-time temporal logic,'' in
  \emph{Proc. 16th Intl. Conf. Hybrid Systems: Computation and Control
  (\nobreak{HSCC}), Philadelphia, PA, U.S.A., Apr. 8-11, 2013}.\hskip 1em plus
  0.5em minus 0.4em\relax ACM, 2013, pp. 333--342.

\bibitem{PolaBorriDiBenedetto12}
G.~Pola, A.~Borri, and M.~D. Di~Benedetto, ``Integrated design of symbolic
  controllers for nonlinear systems,'' \emph{IEEE Trans. Automat. Control},
  vol.~57, no.~2, pp. 534--539, 2012.

\bibitem{GirardGoesslerMouelhi16}
A.~Girard, G.~G\"{o}ssler, and S.~Mouelhi, ``Safety controller synthesis for
  incrementally stable switched systems using multiscale symbolic models,''
  \emph{IEEE Transactions on Automatic Control}, vol.~61, no.~6, pp.
  1537--1549, June 2016.

\bibitem{TazakiImura09}
Y.~Tazaki and J.~Imura, ``Discrete-state abstractions of nonlinear systems
  using multi-resolution quantizer,'' in \emph{Proc. 12th Intl. Conf. Hybrid
  Systems: Computation and Control (\nobreak{HSCC}), San Francisco, U.S.A.,
  Apr. 13-15, 2009}, ser. Lect. Notes Computer Science, R.~Majumdar and
  P.~Tabuada, Eds., vol. 5469.\hskip 1em plus 0.5em minus 0.4em\relax Springer,
  2009, pp. 351--365.

\bibitem{MouelhiGirardGossler13}
S.~Mouelhi, A.~Girard, and G.~G\"{o}ssler, ``Cosyma: A tool for controller
  synthesis using multi-scale abstractions,'' in \emph{Proc. 16th Intl. Conf.
  Hybrid Systems: Computation and Control (\nobreak{HSCC}), Philadelphia, PA,
  U.S.A., Apr. 8-11, 2013}.\hskip 1em plus 0.5em minus 0.4em\relax New York,
  NY, USA: ACM, 2013, pp. 83--88.

\bibitem{LeCorroncGirardGoessler13}
E.~{Le Corronc}, A.~Girard, and G.~Goessler, ``Mode sequences as symbolic
  states in abstractions of incrementally stable switched systems,'' in
  \emph{Proc. 52th IEEE Conf. Decision and Control (CDC), Florence, Italy,
  10-13 Dec. 2013}.\hskip 1em plus 0.5em minus 0.4em\relax New York: IEEE,
  2013, pp. 3225--3230.

\bibitem{ZamaniTkachevAbate14}
M.~Zamani, I.~Tkachev, and A.~Abate, ``Bisimilar symbolic models for stochastic
  control systems without state-space discretization,'' in \emph{Proc. 17th
  Intl. Conf. Hybrid Systems: Computation and Control (\nobreak{HSCC}), Berlin,
  Germany, Apr. 15-17, 2014}.\hskip 1em plus 0.5em minus 0.4em\relax New York,
  NY, USA: ACM, 2014, pp. 41--50.

\bibitem{i14symc}
\BIBentryALTinterwordspacing
G.~Reissig and M.~Rungger, ``Feedback refinement relations for symbolic
  controller synthesis,'' in \emph{Proc. IEEE Conf. Decision and Control (CDC),
  Los Angeles, CA, U.S.A., 15-17 Dec. 2014}.\hskip 1em plus 0.5em minus
  0.4em\relax New York: IEEE, 2014, pp. 88--94.
\BIBentrySTDinterwordspacing

\bibitem{Djokovic70}
D.~{\v{Z}}. Djokovi{\'c}, ``Note on nonnegative matrices,'' \emph{Proc. Amer.
  Math. Soc.}, vol.~25, pp. 80--82, 1970.

\bibitem{London71}
D.~London, ``On matrices with a doubly stochastic pattern,'' \emph{J. Math.
  Anal. Appl.}, vol.~34, pp. 648--652, 1971.

\bibitem{i15gridc}
\BIBentryALTinterwordspacing
M.~Rungger, A.~Weber, and G.~Reissig, ``State space grids for low complexity
  abstractions,'' in \emph{Proc. IEEE Conf. Decision and Control (CDC), Osaka,
  Japan, 15-18 Dec. 2015}.\hskip 1em plus 0.5em minus 0.4em\relax New York:
  IEEE, 2015, pp. 6139--6146.
\BIBentrySTDinterwordspacing

\bibitem{RockafellarWets09}
R.~T. Rockafellar and R.~J.-B. Wets, \emph{Variational analysis}, ser.
  Grundlehren der Mathematischen Wissenschaften.\hskip 1em plus 0.5em minus
  0.4em\relax Berlin: Springer-Verlag, 1998, vol. 317, 3rd corr printing 2009.

\bibitem{Filippov88}
A.~F. Filippov, \emph{Differential equations with discontinuous righthand
  sides}, ser. Mathematics and its Applications (Soviet Series).\hskip 1em plus
  0.5em minus 0.4em\relax Kluwer Academic Publishers Group, Dordrecht, 1988,
  vol.~18, translated from the Russian.

\bibitem{i11abs}
G.~Rei{\ss}ig, ``Computing abstractions of nonlinear systems,'' \emph{IEEE
  Trans. Automat. Control}, vol.~56, no.~11, pp. 2583--2598, Nov. 2011.

\bibitem{OrtegaRheinboldt00}
J.~M. Ortega and W.~C. Rheinboldt, \emph{Iterative solution of nonlinear
  equations in several variables}, ser. Classics in Applied Mathematics.\hskip
  1em plus 0.5em minus 0.4em\relax Philadelphia, PA: Society for Industrial and
  Applied Mathematics (SIAM), 2000, vol.~30, reprint of the 1970 original.

\bibitem{Solodov09}
M.~V. Solodov, ``Global convergence of an {SQP} method without boundedness
  assumptions on any of the iterative sequences,'' \emph{Math. Programming},
  vol. 118, no. 1, Ser. A, pp. 1--12, 2009.

\bibitem{Tarjan72}
R.~Tarjan, ``Depth first search and linear graph algorithms,'' \emph{SIAM
  Journal on Computing}, 1972.

\bibitem{BermanPlemmons94}
A.~Berman and R.~J. Plemmons, \emph{Nonnegative matrices in the mathematical
  sciences}, ser. Classics in Applied Mathematics.\hskip 1em plus 0.5em minus
  0.4em\relax Philadelphia, PA: Society for Industrial and Applied Mathematics
  (SIAM), 1994, vol.~9, revised reprint of the 1979 original.

\bibitem{Kato82}
T.~Kato, \emph{A short introduction to perturbation theory for linear
  operators}.\hskip 1em plus 0.5em minus 0.4em\relax New York: Springer-Verlag,
  1982.

\bibitem{GraichenTreuerZeitz07}
K.~Graichen, M.~Treuer, and M.~Zeitz, ``Swing-up of the double pendulum on a
  cart by feedforward and feedback control with experimental validation,''
  \emph{Automatica J. IFAC}, vol.~43, no.~1, pp. 63--71, 2007.

\end{thebibliography}
\ifx\DraftVersion\undefined%
\else%
\newpage
\ifCLASSOPTIONtwocolumn\noindent\onecolumn\large\fi
\include{ReplyToReviewers}
\fi
\end{document}